\documentclass{article}
\usepackage [utf8] {inputenc}
\usepackage{mathtools}
\usepackage{amsthm}
\usepackage{amssymb}

\newtheorem{theorem}{Theorem}

\newtheorem{lemma}{Lemma}
\newtheorem{corollary}{Corollary}

\DeclareMathOperator{\Id}{Id}
\newcommand{\diamd}{\overrightarrow{\mathrm{diam}}}
\DeclareMathOperator{\F}{\mathbb{F}}

\title{On a polynomial bound for the orbital diameter of primitive affine groups}

\author{Saveliy V. Skresanov}
\date{}
\begin{document}
\maketitle

\begin{abstract}
	Let \( VG \) be a finite primitive affine permutation group, where \( V \) is a vector space of
	dimension \( d \) over the prime field \( \F_p \) and \( G \) is an irreducible linear group on \( V \).
	We prove that if \( p \) divides \( |G| \), then the diameters of all nondiagonal orbital graphs of \( VG \) are at most~\( 9d^3 \).
	This improves an earlier exponential bound by A.~Mar\'oti and the author.
\end{abstract}

\section{Introduction}

Given a finite permutation group \( H \) acting on the set \( \Omega \) one associates with it the set of
\emph{orbital graphs}, that is, orbits of \( H \) on the set of ordered pairs~\( \Omega \times \Omega \). Every such graph
is oriented and orbital graphs with vertices lying in the diagonal \( \{ (x, x) \mid x \in \Omega \} \) are called \emph{diagonal}.

Higman's criterion~\cite{Higman} implies that the group \( H \) is primitive if and only if each nondiagonal orbital graph
is (strongly) connected, in particular, each nondiagonal orbital graph \( \Gamma \) has correctly defined diameter.
By \emph{directed diameter} we will mean the maximal distance between a pair of vertices of~\( \Gamma \)
where paths must preserve edge orientation, while \emph{undirected diameter} will be the similar notion but where edge orientations are ignored.

Recall that a finite primitive permutation group is \emph{affine}, if it decomposes into a semidirect product \( VG \)
of a regular subgroup \( V \) isomorphic to a vector space of dimension \( d \) over the prime field \( \F_p \) and an irreducible matrix group \( G \) acting on \( V \).
The elementary abelian subgroup \( V \) can be identified with the set of points, so we may assume that \( VG \) is a permutation group on \( V \).
More precisely, \( VG \) consists of the following permutations
\[ VG = \{ v \mapsto b + vA \mid b \in V,\, A \in G \}, \]
where \( b \) and \( v \) denote vectors from \( V \), and \( A \) denotes a matrix from \( G \) (notice that we use the row-vector convention).
It follows immediately from this description that the orbital graphs of \( VG \) are Cayley graphs over \( V \) with orbits of \( G \) on \( V \)
playing the roles of connection sets; the zero orbit \( \{ 0 \} \) corresponds to the diagonal orbital graph under this identification.
We will write \( \diamd(V, G) \) for the maximum of directed diameters of orbital graphs of the affine group~\( VG \).

Bounding orbital diameters of primitive permutation groups has applications to model theory, see~\cite{LMT}, where the authors
described infinite classes of finite permutation groups of bounded orbital diameter. One is also interested in providing upper and lower
bounds for orbital diameters of specific classes of permutation groups, see~\cite{Sh} for the symmetric and alternating groups and~\cite{Re}
for the groups of diagonal type. Here the class of affine groups deserves special interest as the description of affine groups with bounded
orbital diameter from~\cite{LMT} is still incomplete. We also refer the reader to the introduction of~\cite{AS} for a more thorough overview
of related problems, in particular, to the connections of orbital diameters with some questions from number theory and combinatorics.

In~\cite[Theorem~1.1]{AS} A.~Mar\'oti and the author proved that the diameters of orbital graphs of a finite primitive affine group \( VG \)
are bounded exponentially in terms of \( \dim V \) and \( \log |V| / \log |G| \). In fact, if \( G \) contains a composition factor,
isomorphic to a finite simple group of Lie type in characteristic \( p \), then one can be more precise and bound the diameters in terms of~\( \dim V \) only by~\( 2^{22(\dim V)^3} \).
The main result of this note shows that if \( p \) divides \( |G| \), one can obtain a polynomial bound on the diameters.

\begin{theorem}\label{main}
	Let \( VG \) be a finite primitive affine permutation group, where \( V \simeq \F_p^d \) for a prime \( p \),
	and \( G \) is an irreducible linear group on \( V \) such that \( p \) divides \( |G| \).
	Then \( \diamd(V, G) \leq 9d^3 \).
\end{theorem}

Clearly the main result applies when \( G \) contains a composition factor isomorphic to a finite simple group of Lie type in characteristic~\( p \).
In this case~\cite{AS} provided an upper bound on the undirected diameter only, so our main result is also a qualitative improvement.

The main obstacle in improving the upper bound when \( p \) does not divide \( |G| \) is the Waring problem in finite fields.
Any bound on the orbital diameter in this case would imply an upper bound for that problem,
while the best known bound for the Waring problem in finite fields is exponential, see~\cite{CC}.

We note that an upper bound in Theorem~\ref{main} is close to being optimal, but there is still some room for improvement.
For an odd prime \( p \) the wreath product of cyclic groups \( G = C_2 \wr C_p \) acts irreducibly and imprimitively on \( V = \F_p^p \),
and by~\cite[Proposition~7.1]{AS} has diameter \( \diamd(V, G) = p(p-1)/2 \). It follows that the upper bound on the diameter must grow
at least quadratically in~\( d \). Nevertheless, it is still possible that the best upper bound should be quadratic in~\( d \),
and as one will see from the proof, if \( p \geq 9d^2 \) then we have a bound of the form \( 4d^2 \log d \), supporting this view.

\section{Proof of the main result}

For two subsets \( \Delta, \Pi \) of a vector space \( V \), let \( \Delta + \Pi \) denote
the set of pairwise sums. For a positive integer \( m \) we write \( m \cdot \Delta \) for the \( m \)-fold sumset \( \Delta + \dots + \Delta \).
As orbital graphs of a primitive affine group \( VG \) are Cayley graphs, the orbital diameter can be computed
in terms of sumsets by the following formula (cf.~\cite[Formula~(1)]{AS}):
\[ \diamd(V, G) = \min \{ m \in \mathbb{N} \mid m \cdot (\mathcal{O}_i \cup \{ 0 \}) = V \text{ for all } i = 1, \dots, r \}, \tag{$\star$} \]
where \( \mathcal{O}_1, \dots, \mathcal{O}_r \) are all nonzero orbits of \( G \) on \( V \). The formula is well-defined,
since the irreducibility of \( G \) implies that each \( \mathcal{O}_i \), \( i = 1, \dots, r \), spans \( V \).

The following lemma is a slight generalization of the result~\cite[Lemma~2.2]{AS}.

\begin{lemma}\label{subspace}
	Let \( G \) be an irreducible linear group on \( V \simeq \F_p^d \) and
	let \( \Delta \) be a \( G \)-invariant subset of \( V \). If for some \( m \geq 1 \)
	the sumset \( m \cdot \Delta \) contains a nontrivial line \( b + \F_p v \) for \( b \in V \), \( v \neq 0\), then \( dm \cdot \Delta = V \).
\end{lemma}
\begin{proof}
	Since \( G \) acts irreducibly on \( V \), the sum of one-dimensional subspaces \( \sum_{g \in G} \F_p v^g \) equals \( V \).
	Therefore there exist elements \( g_1, \dots, g_d \in G \) such that \( \F_p v^{g_1} + \dots + \F_p v^{g_d} = V \).
	As \( (b + \F_p v)^{g_i} \subseteq m \cdot \Delta \) for all \( i = 1, \dots, d \), we have
	\begin{multline*}
		dm \cdot \Delta \supseteq (b+\F_p v)^{g_1} + \dots + (b+\F_p v)^{g_d} = b^{g_1} + \dots + b^{g_d} + \F_p v^{g_1} + \dots + \F_p v^{g_d} =\\
		= b^{g_1} + \dots + b^{g_d} + V = V.
	\end{multline*}
	The claim is proved.
\end{proof}

\begin{corollary}\label{trivial}
	\( \diamd(V, G) \leq d(p-1) \).
\end{corollary}
\begin{proof}
	If \( \mathcal{O} \) is a nonzero orbit of \( G \) on \( V \), then \( (p-1) \cdot (\mathcal{O} \cup \{ 0 \}) \) contains
	a nontrivial line (passing through zero). Then \( d(p-1) \cdot (\mathcal{O} \cup \{ 0 \}) = V \) and the claim follows from Formula~\( (\star) \).
\end{proof}

The next lemma is a consequence of a result of Linnik~\cite{Linnik} about simultaneous solutions of diagonal forms.
The proof relies on Weil's bound for exponential sums, and here we cite a version with effective constants due to Karatsuba.
Let \( \log k \) denote the natural logarithm of~\( k \).

\begin{lemma}[{\cite[Lemma~2]{Kar}}]\label{forms}
	For any \( k \geq 1 \), \( b_1, \dots, b_k \in \F_p \), \( m > 4k \log k \) and a prime \( p \) such that \( p \geq 9k^2 \)
	the system
	\[
		\begin{cases}
			x_1 + \dots + x_m = b_1,\\
			x_1^2 + \dots + x_m^2 = b_2,\\
			\vdots\\
			x_1^k + \dots + x_m^k = b_k.
		\end{cases}
	\]
	has a solution \( x_1, \dots, x_m \in \F_p \).
\end{lemma}

Now we are ready to obtain the main result.
\smallskip

\noindent\textit{Proof of Theorem~\ref{main}.}
If \( p < 9d^2 \), then \( \diamd(V, G) \leq 9d^3 \) by Corollary~\ref{trivial}, so from now on we may assume \( p \geq 9d^2 \).

As \( |G| \) is divisible by \( p \), the group \( G \) has an element \( A \) of order \( p \).
This element acts on the vector space \( V \) defined over the field of order~\( p \), so all eigenvalues of \( A \)
are equal to \( 1 \). Let \( k \) be the degree of the minimal polynomial of \( A \), that is, the smallest positive integer such that \( (A - \Id)^k = 0 \),
where \( \Id \) is the identity transformation. Clearly \( 2 \leq k \leq d \) as \( A \neq \Id \).

Let \( \mathcal{O} \) be a nonzero orbit of \( G \) on \( V \). Since the kernel of \( (A - \Id)^{k-1} \) is a proper subspace of \( V \)
and \( \mathcal{O} \) spans the whole of \( V \), we can choose a vector \( v \in \mathcal{O} \) such that \( v(A - \Id)^{k-1} \neq 0 \).
Now for an arbitrary nonnegative integer \( x \) consider the binomial formula
\[ A^x = (A - \Id + \Id)^x = \sum_{i = 0}^x \binom{x}{i} (A - \Id)^i = \sum_{i = 0}^{k-1} \binom{x}{i} (A - \Id)^i. \]
In the last equality we replaced the upper limit in the summation by \( k-1 \) as \( (A - \Id)^i = 0 \) for all \( i \geq k \).
Notice that one can view the binomial coefficients
\[ \binom{x}{i} = \frac{x(x-1) \dots (x-i+1)}{i!} \]
as polynomials of degree \( i \) in \( x \), and since \( i < k \leq d < p \) we may safely assume that these polynomials have coefficients in \( \F_p \).
Since \( A^p = \Id \), we may identify \( x \) with its residue modulo \( p \) as well.

Now, let \( \alpha \in \F_p \) be arbitrary. By Lemma~\ref{forms}, for any \( m > 4(k-1) \log (k-1) \)
there exist \( x_1, \dots, x_m \in \F_p \) satisfying the following system:
\[
	\begin{cases}
		x_1 + \dots + x_m = 0,\\
		\vdots\\
		x_1^{k-2} + \dots + x_m^{k-2} = 0,\\
		x_1^{k-1} + \dots + x_m^{k-1} = \alpha \cdot (k-1)!.
	\end{cases}
\]
Note that if \( k = 2 \) we have only one equation with the right hand side \( \alpha \), and we may take \( m = 1 \).
Set \( m = \max \{ 1,\, \lceil 4(k-1) \log (k-1) \rceil \} \).

We get the following expression for the sum of binomial coefficients after expanding the brackets and rearranging:
\[ \binom{x_1}{i} + \dots + \binom{x_m}{i} = \frac{x_1^i + \dots + x_m^i}{i!} + c_{i-1}(x_1^{i-1} + \dots + x_m^{i-1}) + \dots + c_1(x_1 + \dots + x_m), \]
where \( c_1, \dots, c_{i-1} \in \F_p \) do not depend on \( x_1, \dots, x_m \). It follows that
\[ \sum_{j = 1}^m \binom{x_j}{i} = 
   \begin{cases}
	   m, & \text{ for } i = 0,\\
	   0, & \text{ for } 0 < i < k-1,\\
	   \alpha, & \text{ for } i = k-1.
   \end{cases}
\]
We now apply this observation to our formula for the power of the element \( A \):
\[ A^{x_1} + \dots + A^{x_m} = \sum_{i = 0}^{k-1} (A - \Id)^i \sum_{j = 1}^m \binom{x_j}{i} = \Id \cdot m + (A - \Id)^{k-1} \cdot \alpha. \]
Elements \( A^{x_1}, \dots, A^{x_m} \) lie in \( G \), so vectors \( v A^{x_1}, \dots, v A^{x_m} \) lie in \( \mathcal{O} \). Thus
\[ m \cdot \mathcal{O} \ni v A^{x_1} + \dots + v A^{x_m} = v \cdot m + [v(A - \Id)^{k-1}] \cdot \alpha. \]
Since \( \alpha \in \F_p \) was arbitrary and \( v(A - \Id)^{k-1} \) is a nonzero vector, the sumset \( m \cdot \mathcal{O} \) contains a line.
By Lemma~\ref{subspace}, we have \( dm \cdot \mathcal{O} = V \) and as \( m \leq 4d^2 \) we get \( 4d^3 \cdot \mathcal{O} = V \).
Since the nonzero orbit \( \mathcal{O} \) was arbitrary, Formula~\( (\star) \) implies \( \diamd(V, G) \leq 4d^3 \), as claimed.

\section{Acknowledgments}

The author expresses his gratitude to prof.\ A.~Mar\'oti and A.V.~Vasil'ev for helpful comments and remarks.

The project leading to this application has received funding from the European Research Council (ERC)
under the European Union’s Horizon 2020 research and innovation programme (grant agreement No 741420).

\medskip

\noindent\emph{Alfr\'ed R\'enyi Institute of Mathematics,\\ Re\'altanoda utca 13-15, H-1053, Budapest, Hungary}

\noindent\emph{E-mail address: skresan@renyi.hu}

\end{document}